\documentclass[12pt,leqno,fleqn]{amsart}  
\usepackage{amsmath,amstext,amsthm,amssymb,amsxtra}
\usepackage[top=1.5in, bottom=1.5in, left=1.25in, right=1.25in]	{geometry}
\usepackage[normalem]{ulem}
\usepackage{txfonts} % pxfonts txfonts 
\usepackage[T1]{fontenc}
\usepackage{lmodern}
\usepackage{tikz}\usetikzlibrary{arrows}

\usepackage{euler}   % better than the option below

\makeatletter
\DeclareRobustCommand\widecheck[1]{{\mathpalette\@widecheck{#1}}}
\def\@widecheck#1#2{%
    \setbox\z@\hbox{\m@th$#1#2$}%
    \setbox\tw@\hbox{\m@th$#1%
       \widehat{%
          \vrule\@width\z@\@height\ht\z@
          \vrule\@height\z@\@width\wd\z@}$}%
    \dp\tw@-\ht\z@
    \@tempdima\ht\z@ \advance\@tempdima2\ht\tw@ \divide\@tempdima\thr@@
    \setbox\tw@\hbox{%
       \raise\@tempdima\hbox{\scalebox{1}[-1]{\lower\@tempdima\box
\tw@}}}%
    {\ooalign{\box\tw@ \cr \box\z@}}}
\makeatother
\usepackage{mathtools}
\mathtoolsset{showonlyrefs,showmanualtags}

\usepackage{hyperref} %,hypertexnames=false,colorlinks,[pagebackref]
\hypersetup{
    colorlinks=true,       % false: boxed links; true: colored links
    linkcolor=blue,          % color of internal links
    citecolor=magenta,        % color of links to bibliography
    filecolor=magenta,      % color of file links
    urlcolor=cyan           % color of external links
}

\usepackage[msc-links]{amsrefs}

\theoremstyle{plain} % definition 
\newtheorem{lemma}[equation]{Lemma} 
\newtheorem{proposition}[equation]{Proposition} 
\newtheorem{theorem}[equation]{Theorem}

\theoremstyle{definition}

\theoremstyle{remark}

\numberwithin{equation}{section}

\title[Sparse  Bounds for Discrete Maximal Operators] { Sparse Bounds for the \\ Discrete Spherical Maximal Functions}
\author[R. Kesler]{Robert Kesler}

\address{ School of Mathematics, Georgia Institute of Technology, Atlanta GA 30332, USA}
\email{robertmkesler@gmail.com}

\author[M. T. Lacey] {Michael T. Lacey}   %  can use \and  

\address{ School of Mathematics, Georgia Institute of Technology, Atlanta GA 30332, USA}
\email {lacey@math.gatech.edu}
\thanks{Research supported in part by grant  from the US National Science Foundation, DMS-1600693 and the 
Australian Research Council ARC DP160100153.}

 \author[D. Mena]{Dar\'io Mena} 
\address{Escuela de Matem\'atica, Universidad de Costa Rica, San Jos\'e, Costa Rica}
\email {dario.menaarias@ucr.ac.cr}
\thanks{Research supported by project 821-B8-287, CIMPA, Escuela de Matem\'atica, UCR.}

\begin{document}
\begin{abstract}
We prove sparse bounds for   the  spherical   maximal operator of Magyar, Stein and Wainger. 
The bounds are conjecturally sharp, and contain an endpoint estimate.    
The new method of  proof is inspired by ones by Bourgain and  Ionescu,   is very efficient, and has not been used in the proof of sparse bounds before.  
The Hardy-Littlewood Circle method is used to decompose the multiplier into major and minor arc components.  The efficiency arises as one 
only needs a single estimate on each  element of the decomposition.
\end{abstract}

	\maketitle  
\tableofcontents 
%%%%%%%%%%%%%%%%%%%%%%%%%%%%%% SECTION  SECTION SECTION
%%%%%%%%%%%%%%%%%%%%%%%%%%%%%% SECTION  SECTION SECTION 
\section{Introduction} %\label{s:}

 Let  $ \mathcal A _{\lambda } f = d \sigma _{\lambda } \ast f $ where 
 $ d \sigma _{\lambda }$ is a uniform unit mass spherical measure on a sphere of radius $ \lambda $ in 
$ \mathbb R ^{d}$, for $ d \geq 3$.
Set the Stein spherical maximal operator to be 
\begin{equation}\label{e:AA}
\mathcal A f (x)  = \sup _{\lambda >0}  \mathcal A _{\lambda } f , 
\end{equation}
where $ f$ is a non-negative compactly supported and bounded function.    
We are interested in sparse bound for the maximal function. In the continuous case, this estimate holds, 
and is sharp, up to the boundary. 

%%%%%%%%%%%%%%%%%%%%%%%%%%%%%% THEOREM THEOREM THEOREM
\begin{theorem}\label{t:R} \cite{170208594L} Let $ d\geq 3$ and set $ \mathbf R_d$ to be  the polygon with vertices 
$ R_0 =   (\frac{d-1}d, \frac{1}d)$,  $ R_1 = (\frac{d-1}d, \frac{d-1}d)$, $ R_2 = (\frac{d ^2 -d} {d ^2 +1}, \frac{d ^2 -d+2} {d ^2 +1})$,  and 
$R_3 =(0,1)$.  (See Figure~\ref{f:A}.)
Then, for all $ (\frac{1}p , \frac{1}q)$ in the interior $ \mathbf R_d$, we have the sparse bound 
$
\lVert \mathcal A  \rVert _{p,q} < \infty  
$. 

\end{theorem}
%%%%%%%%%%%%%%%%%%%%%%%%%%%%%% THEOREM THEOREM THEOREM

%%%%%%%%%%%%%%% Figure
\begin{figure}
\begin{tikzpicture}[xscale=0.8,yscale=0.8] 
\draw [->]  (-.5, 0) -- (5.5,0) node[below] {$ 1/p$}; 
\draw [->]  (0,-.5) -- (0,5.5) node[left] {$ 1/q$}; 
\draw (4.8,-.25 )  node [below] {$ 1$}  -- (4.8,4.8) -- (-.25,4.8) node [left] {$ 1$} ; 
\draw (3.8,1) --  (3.8, 3.8) node [right] {${}_{R_1}$} 
--  (3.6, 4)  node[above] {${}_{R_2}$} -- (0,4.8)  -- (3.8,1) ;  
\draw (3.8 , .25) -- (3.8,-.25) node [below] {$ \tfrac {d-1}d$}; 
\end{tikzpicture}
\caption{Sparse bounds hold for points $ (1/p, 1/q)$ in the interior of the four sided region $ \mathbf R_d$. 
The point $ R_1$ is $(\frac{d-1}d, \frac{d-1}d)$ and $ R_2 $ is $(\frac{d ^2 -d} {d ^2 +1}, \frac{d ^2 -d+2} {d ^2 +1})$.}
\label{f:A}
\end{figure}
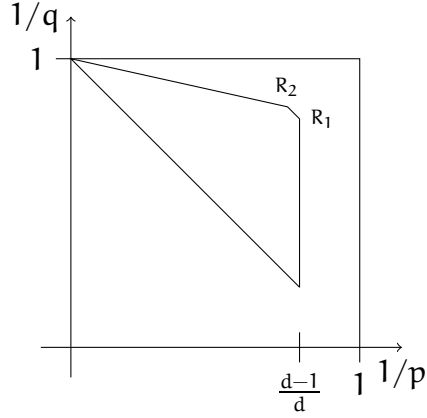
%%%%%%%%%%%%%%% Figure

We set notation for the sparse bounds.  
Call a collection of cubes $ \mathcal S$  in $ \mathbb R ^{n}$ \emph{sparse} if there 
are sets $ \{ E_S  \,:\, S\in \mathcal S\}$  
which are pairwise disjoint,   $E_S\subset  S$ and satisfy $ \lvert  E_S\rvert > \tfrac 14 \lvert  S\rvert  $. 
For any cube $ Q$ and $ 1\leq r < \infty $, set $ \langle f \rangle_ {Q,r} ^{r} = \lvert  Q\rvert ^{-1} \int _{Q} \lvert  f\rvert ^{r}\; dx  $.  Then the $ (r,s)$-sparse form $ \Lambda _{\mathcal S, r,s} = \Lambda _{r,s} $, indexed by the sparse collection $ \mathcal S$ is 
\begin{equation} \label{e:sparsR_def}
\Lambda _{S, r, s} (f,g) = \sum_{S\in \mathcal S} \lvert  S\rvert \langle f  \rangle _{S,r} \langle g \rangle _{S,s}.  
\end{equation}
For a sublinear operator $ T$, we set $ \lVert T \rVert _{r,s}  $ to be the best constant $ C$ in 
the inequality 
\begin{equation*}
\langle T f, g \rangle < C \sup _{\mathcal S} \Lambda _{S, r, s} (f,g) . 
\end{equation*}
We use the same notation for sublinear operators $ T$ acting on functions defined on $ \mathbb Z ^{d}$.

The theorem above refines the well-known $ L ^{p}$-improving properties for the local maximal function 
$ \sup _{1\leq \lambda  \leq 2}  \mathcal A _{\lambda } \ast  f $, 
proved by Schlag \cite{MR1388870} and Schlag and Sogge \cite{MR1432805}.  Also see \cite{MR1949873}.  
The Theorem above  has as immediate corollaries 
(a) vector valued inequalities, and (b) weighted consequences.  Both sets of consequences are the strongest known.  The method of proof  uses the $ L ^{p}$-improving 
inequalities for the spherical maximal function.  That is, the proof is, in some sense, standard, although only recently discovered, and yields the best known information about the mapping properties of the spherical maximal function.

We turn to the setting of discrete spherical averages.  Provided $ \lambda ^2 $ is an integer, and dimension $ d \geq 5$, we can define 
\begin{equation}\label{e:A}
A _{\lambda } f (x) = \lambda  ^{2-d} \sum_{n \in \mathbb Z ^{d} \;:\; \lvert  n\rvert = \lambda  } f (x-n) 
\end{equation}
for functions $ f\in \ell ^2 (\mathbb Z ^{d})$. 
We restrict attention to the case of $ d\geq 5$ as in that case for all $ \lambda ^2 \in \mathbb N $, 
the cardinality of $ \{ n\in \mathbb Z ^{d} \;:\; \lvert  n\rvert = \lambda \}  \simeq \lambda ^{d-2} $.  
Let $ A f = \sup _{\lambda } A _{\lambda } f $, where we will 
always understand that $ \lambda ^2 \in \mathbb N $.    This is the maximal function of 
Magyar \cite{MR1617657} and Magyar, Stein and Wainger \cite{MSW}. 
The following Theorem is the best known extension of the sparse bounds for the continuous spherical maximal function to the discrete 
setting. 

%%%%%%%%%%%%%%%%%%%%%%%%%%%%%% THEOREM THEOREM THEOREM
\begin{theorem}\label{t:Z}
Let $ \mathbf Z_d$ be  the polygon with vertices  
\begin{equation}  \label{e:ZZ}
Z_j = \tfrac{d-4} {d-2} R_j + \tfrac 2{d-2} (\tfrac12, \tfrac12), \qquad j=0,1,2,
\end{equation}
and $ Z_3 = (0,1)$.  (See Figure~\ref{f:B}.)
There holds: 
%%  ENUMERATE
\begin{enumerate}
\item    For all $ (\frac{1}p , \frac{1}q)$ in the interior $ \mathbf Z_d$, we have the sparse bound 
$
\lVert  A f  \rVert _{p,q} < \infty 
$. 

\item  With $ f = \mathbf 1_{F}$ and $ g = \mathbf 1_{G}$, there holds 
\begin{equation} \label{e:Z2}
\langle A \mathbf 1_{F} , \mathbf 1_{G}  \rangle \lesssim \sup _{\mathcal S} 
\Lambda _{\mathcal S, \frac d{d-2}, \frac d{d-2}} (\mathbf 1_{F}, \mathbf 1_{G}).  
\end{equation}
\end{enumerate}
%% ENUMERATE

\end{theorem}
%%%%%%%%%%%%%%%%%%%%%%%%%%%%%% THEOREM THEOREM THEOREM

By direct computation, $ Z_0= (\tfrac{d-2}d, \frac2{d})$, $ Z_1 = (\tfrac{d-2}d,\tfrac{d-2}d)$ and 
\begin{equation}\label{e:Quad}
Z_2 = \bigl(\tfrac{ d ^{3} - 4 d ^2 +4d +1} {d ^{3} -2 d ^2 + d-2} ,  \tfrac{ d ^{3} - 4 d ^2 +6d-7} {d ^{3} -2 d ^2 + d-2}\bigr) . 
\end{equation}
The sparse bound near the point $ (\tfrac{d-2}d, \frac2{d})$ implies the maximal inequality of 
Magyar, Stein and Wainger \cite{MSW}, namely that $ A \;:\; \ell ^{p} (\mathbb{Z} ^{d}) \to \ell ^{p} (\mathbb{Z} ^{d})$, 
for $ p> \frac d{d-2}$.  The sparse bound \eqref{e:Z2} requires that both functions be indicator sets, 
and so is of restricted weak type.  
It implies the restricted weak type inequality of Ionescu \cite{I}.  
These inequalities imply a wide range of weighted and vector valued inequalities, all of which are new. 
See the applications of the sparse bound in the continuous case in \cite{170208594L}.

%%%%%%%%%%%%%%% Figure
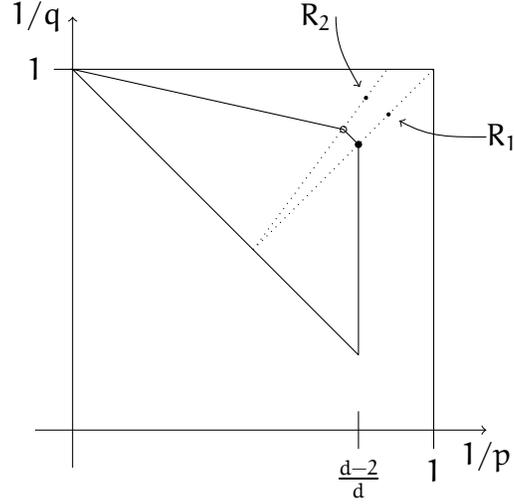
\begin{figure}
\begin{tikzpicture}
\draw [->]  (-.5, 0) -- (5.5,0) node[below] {$ 1/p$}; 
\draw [->]  (0,-.5) -- (0,5.5) node[left] {$ 1/q$}; 
\filldraw (3.8,3.8) circle (.1em) ;  
\draw (3.6, 4)  circle (.1em) ;  

\draw[dotted] (2.4,2.4) -- (4.8,4.8); 
\draw[dotted] (2.4,2.4) -- (4.18,4.8); 
 
\filldraw (4.2,4.2) circle (.05em) node (R1) {};  
\filldraw (3.9, 4.42)  circle (.05em) node (R2) {};  

\draw (3.8, 1.0) -- (3.8,3.8) -- (3.6, 4) -- (0,4.8) -- (3.8, 1);  
\draw (3.8 , .25) -- (3.8,-.25) node [below] {$ \tfrac {d-2}d$}; 
\draw (4.8,-.25 )  node [below] {$ 1$}  -- (4.8,4.8) -- (-.25,4.8) node [left] {$ 1$} ; 

\draw[->]    (5.5, 3.9) node {\ \ \ $ R_1$}   to [out = 180 , in = -30] (R1) ; 
\draw[->]    (3.5, 5.5) node {$ R_2$\ \ \  \ \ }   to [out = -90 , in = 120] (R2) ; 
\end{tikzpicture}
\caption{Sparse bounds for the discrete spherical maximal function 
hold for points $ (1/p, 1/q)$ in the interior of the four sided figure above.
The dotted lines pass through the points $ (1/2,1/2)$ and the points $ R_1$ and $ R_2$ 
of Figure~\ref{f:A}.  Circles along these lines are the points $ Z_1$ and $ Z_2$.  
The  restricted weak-type sparse bound \eqref{e:Z2} holds at the filled in circle, $Z_1 = (\frac d {d-2},\frac d {d-2})$.  
} 
\label{f:B}
\end{figure}
%%%%%%%%%%%%%%% Figure

\medskip  

The discrete spherical maximal function $ \ell ^{p}$ bounds were established by Magyar, Stein and Wainger 
\cite{MSW}, with an endpoint restricted weak-type estimate proved by Ionescu \cite{I}.  
The discrete $ \ell ^{p}$-improving inequalities have only recently been investigated.  
The case of a fixed radius was addressed, independently, in   \cites{180409260H,180409845}. 
Spherical maximal functions, restricting to lacunary   and super lacunary cases require different techniques 
\cites{160904313,180803822,181012344}.  
Robert Kesler established sparse bounds for the discrete case in \cites{180509925,180906468}.
This paper extends and simplifies those arguments.  

It is very tempting to conjecture that our sparse bounds form the sharp range, up to the endpoints. 
One would expect that certain kinds of natural examples would demonstrate this. 
But examples are much harder to come by in the discrete setting.  We return to this in \S~\ref{s:counter}.

The argument in this paper is elegant, especially if one restricts attention to the endpoint estimate 
\eqref{e:Z2}. And much simpler than the arguments in \cites{180509925,180906468}. 
It proceeds by decomposing the maximal function into a series of terms, guided by 
the Hardy-Littlewood circle method decomposition developed by Magyar, Stein and Wainger \cite{MSW}. 
The decomposition has many parts, as indicated in Figure~\ref{f:Fulltree} and Figure~\ref{f:teriary}.  
But, for each part of the decomposition, we need only one estimate, either an $ \ell ^2 $ estimate, 
or an endpoint estimate.  Roughly speaking, one uses either a `high frequency' $ \ell ^2 $ estimate, or a `low frequency' 
inequality, in which one compares to smoother averages.  Interestingly, the notion of `smoother averages' varies.     
The  argument of Ionescu combined with the sparse perspective yields a powerful inequality.  
Notations and conventions will be established in this section, and used throughout the paper.

%%%%%%%%%%%%%%%%%%%%%%%%%%%%%% SECTION  SECTION SECTION
%%%%%%%%%%%%%%%%%%%%%%%%%%%%%% SECTION  SECTION SECTION 
\section{Proof of   the Sparse Bounds inside the polygon $ \mathbf Z_d$} %\label{s:}
A sparse bound is typically proved by recursion. 
So, the  main step is to prove the recursive statement. 
To do this, we fix a large dyadic cube $ E$, functions $ f = \mathbf 1_{F}$ and $ g= \mathbf 1_{G}$ 
supported on $ E$.  We say that $ \tau \;:\; E \to \{ 1 ,\dotsc, \ell E\}$ is 
an \emph{admissible stopping time} if for any subcube $ Q\subset E$ with 
$ \langle f \rangle _{Q} > C \langle f \rangle _{E}$, for some large constant $C$ to be chosen later, we have 
$ \min _{x\in Q} \tau (x) > \ell Q$.  

%%%%%%%%%%%%%%%%%%%%%%%%%%%%%% LEMMA LEMMA LEMMA
\begin{lemma}\label{l:recurse}
Let $ (\tfrac 1p, \tfrac 1q)$ be in the interior of $ \mathbf Z_d$.  
For any dyadic cube, functions 
 $ f = \mathbf 1_{F}$ and $ g= \mathbf 1_{G}$ 
supported on $ E$, and any admissible stopping time $ \tau $, there holds 
\begin{equation}\label{e:recurse}
\lvert  E \rvert ^{-1}  \langle A _{\tau } f , g  \rangle \lesssim \langle f \rangle _{E} ^{1/p} 
\langle g \rangle _{E} ^{1/q}. 
\end{equation}

\end{lemma}
%%%%%%%%%%%%%%%%%%%%%%%%%%%%%% LEMMA LEMMA LEMMA

We complete the proof of the main result. 

%%%%%%%%%%%%%%%%%%%%%%%%%%%%%% PROOF PROOF PROOF
\begin{proof}[Proof of First Part of Theorem~\ref{t:Z}]  We can assume that there is a fixed dyadic cube $ E$ 
so that $ f = \mathbf 1_{F}$ is supported on cube $ 3E$, and $ g= \mathbf 1_{G} $ is 
supported on $ E$.   Let $\mathcal Q_E$ be the maximal dyadic subcubes of $E$ for which 
$ \langle f \rangle _{3Q} > C \langle f \rangle _{3E}$, for a large constant $C$.  
Observe that we have, for an appropriate choice of admissible $ \tau (x)$,  
\begin{align*}
\Bigl\langle \sup _{ \lambda \leq \ell (E)} A _{\lambda } f, g \Bigr\rangle 
& \leq  \langle   A _{\tau  } f , g \rangle 
+ \sum_{Q\in \mathcal Q_E} 
\Bigl\langle \sup _{  \lambda  \leq \ell (Q)} A _{\lambda } (f \mathbf 1_{3Q}) , g \mathbf 1_{Q}\Bigr\rangle .  
\end{align*}
The first term is controlled by \eqref{e:recurse}.  For appropriate constant $C \simeq 3^d$, we have 
\begin{equation*}
\sum _{Q\in \mathcal Q_E} \lvert  Q\rvert \leq \tfrac 14 \lvert E\rvert . 
\end{equation*}  
We can  clearly recurse on the second term above to construct our sparse bound.  
This proves a sparse bound for all indicator functions in the interior of $ \mathbf Z_d$. 

Sparse bounds for indicator functions in an open set self-improve to sparse bounds for functions. 
We give the details in the last section, see Lemma~\ref{l:interpolate}.

\end{proof}
%%%%%%%%%%%%%%%%%%%%%%%%%%%%%% PROOF PROOF PROOF

We use the corresponding recursive  inequality for spherical averages on $ \mathbb R ^{d}$. 
  Recall that $ \mathcal A _{\lambda }$ is the continuous spherical average. 

%%%%%%%%%%%%%%%%%%%%%%%%%%%%%% LEMMA LEMMA LEMMA
\begin{lemma}\label{l:R} \cite{170208594L}*{Lemma 3.4} 
Let $ (\tfrac 1{\bar p}, \tfrac 1{\bar q})$ be in the interior of $ \mathbf R_d$.  
For any dyadic cube $ E$, functions 
 $  \phi  =  \mathbf 1_{F}$ and $ \gamma = \mathbf 1_{G} $ 
supported on $ E$, and any admissible stopping time $ \tau $, there holds 
\begin{equation}\label{e:R}
\lvert  E \rvert ^{-1}  \langle \mathcal A _{\tau } \phi  , \gamma   \rangle \lesssim \langle \phi  \rangle _{E} ^{1/ \bar p} 
\langle \gamma  \rangle _{E } ^{1/\bar q} .  
\end{equation}
\end{lemma}
%%%%%%%%%%%%%%%%%%%%%%%%%%%%%% LEMMA LEMMA LEMMA

We turn to the proof of Lemma~\ref{l:recurse}.  The restriction to indicator functions will allow us to use interpolation arguments, 
even though our setting has stopping times, and hence is non-linear.  
Let $ L$ be the line through $ (\tfrac12, \tfrac12)$ and $ (\tfrac1p, \tfrac1q)$.  Then, let $ (\tfrac1{ \bar p}, \tfrac1{\bar q})$ be a point 
on  $ L$ that is in the interior of $ \mathbf R_d$, and very close to the boundary.  
(The dashed lines in Figure~\ref{f:B} are examples of the lines $ L$ we are discussing here.) 

This is the point: Fix $ (1/ \bar p, 1/\bar q) \in \mathbf R_d$.  For 
all sufficiently small $ 0< \epsilon < 1$ so that $ (1/ (\bar p + \epsilon ) , 1/(\bar q+ \epsilon )) \in \mathbf R_d$, 
and  integers $  N \in \mathbb N $, we can write $ A _{\tau } f \leq M_1 + M_2$ where 
\begin{align}\label{e:M1}
\lvert  E \rvert ^{-1}  \langle M_1 , g  \rangle &  \lesssim N ^{1+ \epsilon } \langle f \rangle _{E} ^{\frac1{\bar p+ \epsilon }}  \langle g \rangle _{E} ^{\frac{1}{\bar q + \epsilon}}, 
\\ \label{e:M2}
\lvert  E \rvert ^{-1}  \langle M_2 , g  \rangle 
&\lesssim N ^{d \epsilon + \frac{4-d}2 } \langle f \rangle _{E } ^{1/2}  \langle g \rangle _{E} ^{1/2}.  
\end{align}
Implied constants depend upon $\bar p, \bar q$ and $ \epsilon$, but we do not track the dependence.     
Once this is proved, one has 
\begin{equation*}
\lvert  E \rvert ^{-1}  \langle A _{\tau } f , g  \rangle \lesssim 
N ^{1+ \epsilon } \langle f \rangle _{E} ^{\frac1{\bar p+ \epsilon }}  \langle g \rangle _{E} ^{\frac1 {\bar q+ \epsilon } } 
+ N ^{d\epsilon + \frac{4-d}2 } \langle f \rangle _{E } ^{1/2} \langle g \rangle _{E} ^{1/2}. 
\end{equation*}
Choosing $ N$ to minimize the right hand side, and letting $ (1/\bar p, 1/\bar q)$ and $ 0< \epsilon<1 $ vary 
completes the proof.   Indeed, ignoring $ \epsilon $'s, we see that the value of $ p$ is given by 
\begin{equation*}
\tfrac1{p} =\tfrac1{\bar p} + \tfrac2{d-2} \bigl(\tfrac{1}2- \tfrac1{\bar p} \bigr) 
= \tfrac2{d-2} \cdot  \tfrac{1}2 + \tfrac{d-4} {d-2} \cdot  \tfrac1{\bar p}, 
\end{equation*}
Compare this to our description of the extreme points of $ \mathbf Z_d$ in  \eqref{e:ZZ}.  
Thus, our Lemma follows.

\smallskip 
In proving \eqref{e:M1} and \eqref{e:M2}, it suffices, given $ 0< \epsilon < 1$, to prove the 
statement for sufficiently large $ N$.  We will do so for $ N > N_0$, for a 
sufficiently large choice of $ N_0 > 0$.    
Indeed, we find it necessary to use an absorption argument. We show that 
\begin{equation}\label{e:absorb}
A _{\tau } f  \leq M_1 + M_2 + \tfrac 12 A _{\tau } f , 
\end{equation}
where $ M_1 $ and $ M_2$ are as in \eqref{e:M1} and \eqref{e:M2}.  

The reader can consult Figure~\ref{f:Fulltree} for a guide to the argument. 
For a technical reason, we assume that $ F \subset (2 \mathbb Z ^{d}) + \delta _f$, 
and $ G \subset (2 \mathbb Z ^{d}) + \delta _g$. 
Here, $ \delta _f , \delta _g \in \{0,1\} ^{d}$.  
This can be assumed without loss of generality.

\tikzset{
 treenode/.style = {shape=rectangle, rounded corners,
    draw, align=center,
    top color=white, bottom color=blue!10}}

%%%%%%%%%%%%%%% Figure
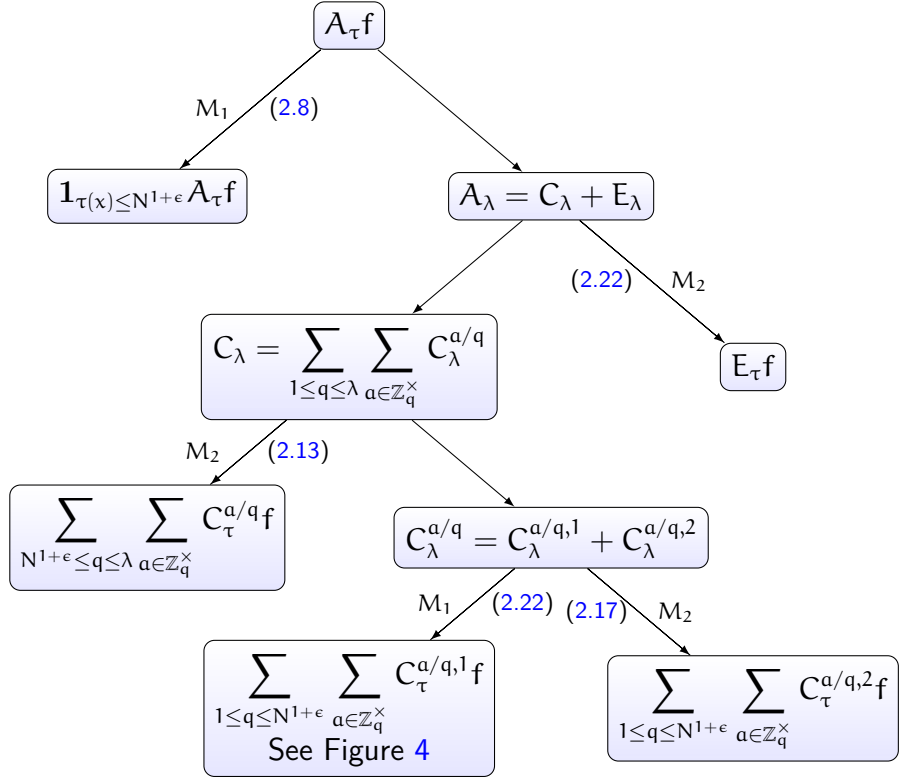
\begin{figure} 
\begin{tikzpicture}[sibling distance=13em, level distance=5.5em, 
    edge from parent/.style = {draw, -latex, font=\scriptsize}]
  \node [treenode] {$ \displaystyle  A _{\tau  } f $}
    child  { node [treenode]{$ \displaystyle  \mathbf 1_{ \tau (x) \leq  N^{1 + \epsilon}}  A _{\tau   } f  $} 
     edge from parent node [left] {$M_1$ \ }
     edge from parent node [right] {\ \eqref{e:zM11} }} 
    child { node [treenode]{$  \displaystyle A _{\lambda } = C _{\lambda } + E_{\lambda }$}
      child { node [treenode]{ $\displaystyle C _{\lambda} = \sum _{1\leq q \leq \lambda } \sum _{a \in \mathbb Z _q ^{\times }} C^{a/q}_\lambda $}
        child { node [treenode]{ $\displaystyle   \sum_{ N^{1 + \epsilon} \leq q \leq \lambda } \sum _{a \in \mathbb Z _q ^{\times }}   C ^{a/q} _{\tau } f  $}  
 		edge from parent node [left] {$M_2$ \ \ }
 		edge from parent node [right] {\ \eqref{e:zM22}}}   
        child { node [treenode] { $\displaystyle  C ^{a/q} _{\lambda } =  C ^{a/q,1} _{\lambda } +  C ^{a/q,2} _{\lambda } $} 
        child { node [treenode] {$\displaystyle  \sum_{1 \leq q \leq  N^{1 + \epsilon}} \sum  _{a \in \mathbb Z _q ^{\times }}   C ^{a/q,1} _{\tau } f   $ \\ See Figure~\ref{f:teriary}}
 		edge from parent node [left] {$M_1$ \ }
 		edge from parent node [right] {\  \eqref{e:M1Q}}} 
 		child { node [treenode] {$\displaystyle    \sum_{1\leq q \leq N^{1 + \epsilon}} \sum _{a\in \mathbb Z_q ^{\times}} 
 C ^{a/q,2} _{\tau } f  $}
   edge from parent node [right] {\ $M_2$}
   edge from parent node [left] {\eqref{e:M23}}
        }}  }
      child { node [treenode] {$\displaystyle      E_{\tau } f   $}
       edge from parent node [right] {\ $M_2$}
       edge from parent node [left] {\eqref{e:M1Q}\  \ }}   };
\end{tikzpicture}
\caption{The flow of the proof of  \eqref{e:recurse}.  The nodes of the tree indicate the different elements 
of the decomposition, and a label on an arrow shows which of  $ M_1$ or $ M_2$ that term contributes to.  Above, $ \lambda $ represents a fixed choice of radius, and $ \tau = \tau (x)$ an admissible  choice of radius. 
For space considerations,   several terms of the form $  e_q (- \lambda ^2 a)$ have been omitted, compare to \eqref{e:Cfull}.  
}  
\label{f:Fulltree}
\end{figure}
%%%%%%%%%%%%%%% Figure

%%%%%%%%%%%%%%%%%%%%%%%%%%%%%% SUBSUBSECTION SUBSUBSECTION SUBSUBSECTION 
\subsubsection*{Small values of $ \tau $}%\label{sss.}
The terms $ M_1$ and $ M_2$ have several components.  
The first contribution to $ M_1$ is the term 
$ M _{1,1} = \mathbf 1_{\tau \leq N ^{1+ \epsilon }} A _{\tau } f $.  
Our verification that $ M _{1,1}$ satisfies \eqref{e:M1} is our first application of 
the $ \mathbb R ^{d}$ inequality \eqref{e:R}.  

We need functions on $ \mathbb R ^{d}$. Take $ \phi (x) = \sum _{n \in \mathbb Z ^{d}} \mathbf 1_{F} (n) 
\mathbf 1_{n + [-1,1) ^{d}}(x)$, and define $ \gamma $ similarly.   By the reduction we made above, these are indicator functions.  
Moreover, if $ \tau $ is admissible stopping time for $ f$, then it is for $ \phi $ as well.  
The inequality 
\eqref{e:R} holds for these two functions on $ \mathbb R ^{d}$.  Then, notice that 
we can compare the discrete and continuous spherical averages as follows.  
\begin{equation}  \label{e:zM11}
A _{\tau  } f (x) \lesssim  \tau  \mathcal A _{\tau  } \phi (x).  
\end{equation}
Therefore, if we require that $ \tau \leq N^{1 + \epsilon} $, we see that \eqref{e:R} implies that $ M _{1,1}$ 
satisfies \eqref{e:M1}.

%%%%%%%%%%%%%%%%%%%%%%%%%%%%%% SUBSUBSECTION SUBSUBSECTION SUBSUBSECTION 
\subsubsection*{The Decomposition}%\label{sss.}
Below, we assume that $ \tau > N ^{1+ \epsilon }$ pointwise. At this point, we need a decomposition of 
$ A _{\lambda } f $ into a family of multipliers. We recall this from Magyar, Stein and Wainger \cite{MSW}.  
Upper case letters denote a convolution operator, and lower case letters denote the corresponding multiplier.  Let $ e (x) = e ^{2 \pi i x}$ and for integers $ q$, $ e_q (x) = e (x/q)$.   
\begin{align}  \label{e:Afull}
A _{\lambda } f &= C _{\lambda } f + E_{\lambda } f , 
\\ \label{e:Cfull}
C _{\lambda } f &= \sum_{1\leq q \leq  \lambda } \sum_{ a \in \mathbb Z _q ^{\times }}  e_q (- \lambda ^2 a)  C ^{a/q} _{\lambda } f ,
\\  \label{e:caq}
c ^{a/q} _{\lambda }(\xi) &= \widehat {C ^{a/q} _{\lambda }}   (\xi ) = \sum_{\ell \in \mathbb Z ^{d}_q} G (a/q, \ell ) \widetilde \psi_q (\xi - \ell /q) 
\widetilde {d \sigma _{\lambda }} (\xi - \ell /q) , 
\\  \label{e:Gauss}
G (a/q, \ell ) &= q ^{-d}\sum_{n \in \mathbb Z _q ^{d}} e_q ( \lvert  n\rvert ^2 a+n \cdot \ell  ). 
\end{align}
The term $ G (a/q, \ell )$ is a normalized Gauss sum.  
In \eqref{e:Cfull}, the sum over $ a \in \mathbb Z _q ^{\times }$ means that $ (a,q)=1$.  
In \eqref{e:caq}, the hat indicates the Fourier transform on $ \mathbb Z ^{d}$, and the notation 
conflates the operator $ C _{\lambda } ^{a/q}$, and the kernel. All our operators are convolution operators 
or maximal operators formed from the same.  
The function $ \psi $ is a Schwartz function on $ \mathbb R ^{d}$ which satisfies 
\begin{equation}\label{e:psi}
 \mathbf 1_{[-\frac{1}{2} , \frac{1}{2}] }(|\xi|) \leq \widetilde \psi  (\xi ) \leq \mathbf 1_{[-1,1] }(|\xi|). 
\end{equation}
Above, $ \widetilde { f }$ denotes the Fourier transform of $f$ on $ \mathbb R ^{d}$, and $ \widetilde  \psi _q (\xi ) = \widetilde \psi (q \xi )$.  
The Fourier transform on $ \mathbb R ^{d}$  of   $ d \sigma _{\lambda }$ is 
$ \widetilde {d \sigma _{\lambda }} $. 
Finally, we will use the notation $ \lambda $ for describing multipliers and so on, and 
using  $\tau $ especially when obtaining estimates. In this way, many supremums will be suppressed 
from the notation.

\bigskip 

%%%%%%%%%%%%%%%%%%%%%%%%%%%%%% SUBSUBSECTION SUBSUBSECTION SUBSUBSECTION 
\subsubsection*{The Error Term $ E _{\lambda }$}%\label{sss.}
The first contribution to $ M_2$ is  $ M _{2,1} = \lvert  E_{\tau } f \rvert $.  
The inequality below is from 
\cite{MSW}*{Prop. 4.1}, and it     implies that $ M _{2,1}$ satisfies \eqref{e:M2} 
since $ \tau > N ^{1+ \epsilon }$.  
\begin{equation} \label{e:zM21} \bigl\lVert \sup _{\Lambda \leq \lambda \leq 2 \Lambda } \lvert  E_{\lambda }  \cdot  \rvert \bigr\rVert_ {2\to 2 }  \lesssim \Lambda ^{\frac{4-d}2}, \qquad \Lambda \geq 1. 
\end{equation}

%%%%%%%%%%%%%%%%%%%%%%%%%%%%%% SUBSUBSECTION SUBSUBSECTION SUBSUBSECTION 
\subsubsection*{Large Denominators}%\label{sss.} 
The second contribution to $ M_2$ is 
\begin{equation} \label{e:zM22}
M _{2,2} =  \Bigl\lvert \sum_{N ^{ 1+ \epsilon } \leq q \leq \tau } e_q (- \lambda ^2 a) C ^{a/q} _{\tau } f \Bigr\rvert. 
\end{equation}
The estimate below  is a result of Magyar, Stein and Wainger \cite{MSW}*{Prop. 3.1}, and it verifies 
that $ M _{2,2}$ satisfies \eqref{e:M2}.  We need only sum it over $ 1\leq a \leq q$, and $ q > N ^{1+ \epsilon }$.  
\begin{equation}  \label{e:MSWfactor}
\lVert \sup _{\lambda > q} \lvert  C ^{a/q} _{\lambda }  f  \rvert \rVert _{2} 
\lesssim q ^{- \frac d2} \lVert f\rVert_2 . 
\end{equation}

%%%%%%%%%%%%%%%%%%%%%%%%%%%%%% SUBSUBSECTION SUBSUBSECTION SUBSUBSECTION 
\subsubsection*{Small Denominators: A Secondary Decomposition}%\label{sss.}
It remains to bound the small denominator case, namely 
\begin{equation*}
\sum_{1\leq q \leq N^{1 + \epsilon}}  \sum_{ a \in \mathbb Z _q ^{\times }} e_q (- \lambda ^2 a)C ^{a/q} _{\tau } f 
\end{equation*}
with further contributions to $ M_1$ and $ M_2$. 
Write 
$ C ^{a/q} _{\tau } = C ^{a/q,1} _{\tau } + C ^{a/q,2} _{\tau }$, with this understanding.  
For an integer $ 1\leq Q \leq N/2$,  and $  Q\leq q < 2 Q$, define 
\begin{equation}\label{e:exC}
 \widehat {C _{\lambda } ^{a/q,1}}  (\xi ) 
=  \sum_{\ell \in \mathbb Z ^{ d}} G (a, \ell, q )\widetilde\psi_ {q} (\xi - \ell /q) \widetilde\psi_ {\lambda Q/N} (\xi - \ell /q)  
\widetilde{d \sigma _{\lambda }} (\xi - \ell /q).  
\end{equation}
Above, we have adjusted the cutoff around each point $ \ell /q \in \mathbb T ^{d}$.  

%%%%%%%%%%%%%%%%%%%%%%%%%%%%%% SUBSUBSECTION SUBSUBSECTION SUBSUBSECTION 
\subsubsection*{Small Denominators: The $ \ell ^2 $ Part}%\label{sss.}
We complete the construction of the term in $ M_2$, \eqref{e:M2},  by showing that 
\begin{equation}   \label{e:Complete}
\bigl\lVert \sup _{N^{1 + \epsilon} \leq \lambda \leq \ell (E)}
 \lvert  
 C ^{a/q,2} _{\lambda } f  
 \rvert \bigr\rVert_2 \lesssim  q ^{-1} N ^{- \frac{d-2}2} \lVert f\rVert_2. 
\end{equation}
It follows  that 
\begin{equation} \label{e:M23}
\Bigl\lVert 
   \sum_{1\leq q \leq N^{1 + \epsilon} }  \sum_{a \in \mathbb Z _q ^{\times }} 
  e_q (- \lambda ^2 a)C ^{a/q,2} _{\tau } f  
\Bigr\rVert_2 \lesssim   N ^{- \frac{d-4}2 + \epsilon} \lVert f\rVert_2. 
\end{equation}
This is the third and final contribution to $ M _{2}$.   We remark that the proof detailed below is a quantitative variant of the proof of Magyar, Stein and Wainger's inequality \eqref{e:MSWfactor}.  
The inequality \eqref{e:Complete} is \cite{I}*{(2.14)}, but we include details here.

Let $ m$ be a smooth function supported on $ [-1/2,1/2] ^{d}$, and let $ T_m$ be the corresponding 
multiplier operator, either on $ \mathbb Z ^{d}$ or $ \mathbb R ^{d}$, with the notation indicating in 
which setting we are considering the multiplier.  

This is a factorization argument from Magyar,  Stein and Wainger \cite{MSW}*{pg. 200}. 
Using the notation of \eqref{e:mzeta} to define $ M _{\psi _{Q/2}}$ below, 
for $ Q\leq q \leq 2Q$, we have  
\begin{align}  \label{e:redux}
\widehat {C ^{a/q,2} _{\lambda } } (\xi ) &= 
\widehat { M _{\psi _{Q/2}, q} } (\xi ) \cdot   \sum_{\ell \in \mathbb Z ^{ d}}  \widetilde\psi _{q} (\xi - \ell /q) (1- \widetilde\psi _{\lambda Q/N} (\xi - \ell /q) )
\widetilde{d \sigma _{\lambda }} (\xi - \ell /q) 
\\
&:= \widehat {  M _{\psi _{q/2}, q}} (\xi )  \cdot  \widehat {C  ^{a/q,3} _{\lambda }}  (\xi ). 
\end{align}
That is, the operator in question factors as $ C ^{a/q,2} _{\lambda } = C ^{a/q,3} _{\lambda } \circ M _{\psi _{Q/2}, q}$.  
Notice that by the Gauss sum estimate, we have 
\begin{equation} \label{e:G2}
\lVert  M _{\psi _{Q/2}, q} \rVert _{ 2 \to 2 } \lesssim Q ^{- \frac{d}2}, 
\end{equation}

In controlling the supremum, we need only consider the supremum over ${C ^{a/q,3} _{\lambda } }$. 
The transference lemma \cite{MSW}*{Cor. 2.1} allows us to estimate this supremum on $ L ^2 (\mathbb R ^{d})$.  We have 
\begin{equation} \label{e:G3}
\bigl\lVert 
{C ^{a/q,3} _{\tau  } } 
\bigr\rVert _{\ell ^2 (\mathbb Z  ^{d}) \to \ell ^{2 } (\mathbb Z ^{d})}
\lesssim 
\bigl\lVert \sup _{ \lambda >0} 
\bigl\lvert  {\Psi _{\lambda ,q}} \cdot  \bigr\rvert
\bigr\rVert _{L^2 (\mathbb R  ^{d}) \to L ^{2 } (\mathbb R ^{d})}, 
\end{equation}
where $ \widetilde { \Psi }_{\lambda ,q}  =   \widetilde\psi_q   (1-\widetilde\psi _{\lambda Q/N})    \widetilde{d \sigma _{\lambda }}$.  
To estimate this last norm on $ L ^2 (\mathbb R ^{d})$, we use this Lemma of Bourgain.
%%%%%%%%%%%%%%%%%%%%%%%%%%%%%% LEMMA LEMMA LEMMA
\begin{lemma}\label{l:B} \cite{MR812567}*{Prop. 2} 
Let $ m$ be a smooth function on $ \mathbb R ^{d}$. We have 
\begin{gather}\label{e:B1}
\Bigl\lVert  \sup _{r>0}  \lvert  T _{m (r \cdot )}  \cdot \rvert  \Bigr\rVert_2  
\lesssim \sum_{j\in \mathbb Z } \alpha _{j} ^{1/2} ( \alpha _{j} ^{1/2}+ \beta _j ^{1/2}) 
\end{gather}
where 
\begin{gather}
 \label{e:B2<} 
\alpha _j = \lVert  \mathbf 1_{2 ^{j} \leq \lvert  \xi \rvert \leq 2 ^{j+1} } m (\xi )\rVert _{\infty } 
\quad \textup{and} \qquad 
\beta  _j = \lVert  \mathbf 1_{2 ^{j} \leq \lvert  \xi \rvert \leq 2 ^{j+1} } \nabla m (\xi ) \cdot \xi \rVert _{\infty }.  
\end{gather}
\end{lemma}
%%%%%%%%%%%%%%%%%%%%%%%%%%%%%% LEMMA LEMMA LEMMA

We bound the right side of \eqref{e:G3}.  Composition with $ T _{\psi _q}$ is uniformly bounded on $ L ^2 $.  The multiplier  in question is then, 
$ m (\xi ) = (1 - \widetilde\psi _{Q/N})  (\xi )\widetilde{d \sigma _1} (\xi )$.  This is identically zero 
for $ \lvert  \xi \rvert \lesssim N/Q $.  That means that for the terms in \eqref{e:B2<}, we need only consider 
$ 2 ^j \gtrsim N/Q\geq 100 $. Recall
the standard stationary phase estimate below.  
\begin{equation} \label{e:stationary}
\lvert  \nabla  \widetilde{d \sigma _1} (\xi ) \rvert + \lvert  \widetilde{d \sigma _1} (\xi ) \rvert \lesssim \lvert  \xi \rvert ^{- \frac{d-1}2}. 
\end{equation} 
Hence, the bound for our multiplier is 
\begin{align*}
\bigl\lVert \sup _{\lambda >0}  
\bigl\lvert  {\Psi _{\lambda ,q}} \cdot  \bigr\rvert
\bigr\rVert _{L^2 (\mathbb R  ^{d}) \to L ^{2 } (\mathbb R ^{d})}
& \lesssim 
\sum_{j \;:\; 2 ^{j} \geq N/Q}  2 ^{- j \frac{d-1}2} \cdot 2 ^{-j \frac{d-3}2}
\\&\lesssim (Q/N) ^{ \frac{d-2}2}  . 
\end{align*}
This estimate combined with \eqref{e:G2} and \eqref{e:G3} complete the proof of \eqref{e:Complete}.  
\bigskip 

%%%%%%%%%%%%%%%%%%%%%%%%%%%%%% SUBSUBSECTION SUBSUBSECTION SUBSUBSECTION 
\subsubsection*{Small Denominators: The Sparse Part}%\label{sss.}
Recalling the notation from \eqref{e:exC}, we turn to   
\begin{equation*}
M_{1,2}f =  \sum_{1\leq q \leq N}  
 \sum_{ a \in \mathbb Z _q ^{\times }}
e_q (- \lambda ^2 a)C ^{a/q,1} _{\tau } f 
\end{equation*}
and show that this term is as in \eqref{e:absorb}.  This is the term in which the absorbing term 
$ \tfrac{1}2 A _{\tau } f$ in \eqref{e:absorb} arises.  
It is also the core of the proof.

Define 
\begin{equation} 
M _{1,Q} f = \sum_{Q\leq q < 2Q} \sum_{ a \in \mathbb Z _q ^{\times }}  e_q (- \lambda ^2 a)C ^{a/q,1} _{\tau } f 
, \qquad  1\leq Q \leq N/2,  
\end{equation}
Above, and below, we will treat $ M _{1,Q} $ as an operator, as we have yet to tease out some of its additional properties.  The main estimate to prove is 
\begin{equation}\label{e:M1Q}
\begin{split}
 M _{1,Q} f & \leq  \overline M _{1,Q} f + N^{- \epsilon} A _{\tau } f , 
 \\
\lvert  E\rvert ^{-1}   \langle     \overline M_{1,Q} f, g \rangle   
& \lesssim   N  ^{1+ \epsilon /2}  \langle f \rangle _{E} ^{\frac1{\bar p + \epsilon} } \langle g \rangle _{E} ^{ \frac 1{\bar q + \epsilon } }. 
 \end{split} 
\end{equation}
 This summed over  dyadic $ 1\leq Q \leq N/2$ to complete the proof of the absorption inequality \eqref{e:absorb}.  
This step requires that $ N $ be sufficiently large,  $ N > \kappa ^{1/\epsilon }$, but that is 
sufficient for our purposes.

We need the estimate \eqref{e:R} on $ \mathbb R ^{d}$. We also need kernel estimates for the operators $ M _{1,Q}$, and for that we require this preparation, 
which has been noted before  \cites{MSW}, \cite{I}*{pg. 1415}. 
For a function $ \zeta  $ with \( \widetilde \zeta \)  supported  on $[-1,1] ^{d}$, define a  family of Fourier multipliers by 
\begin{equation} 
\widehat {M _{\zeta, q}}  (\xi ) = 
\sum_{\ell \in \mathbb Z ^{d}_q} 
G (a/q, \ell ) \widetilde \zeta (\xi - \ell /q).  
\end{equation}
By inspection, the Gauss sum map $ \ell \mapsto G (a/q, \ell )$ is the Fourier transform of $ e_q  (  \lvert  x\rvert  ^2 a ) $ 
as a function on $ \mathbb Z _{q} ^{d}$.  From this, and a routine computation, it follows that 
\begin{equation}\label{e:mzeta}
 M _{\zeta, q} (x) = e_q (  a\lvert  x\rvert ^2)  \zeta (x).  
\end{equation}
(Here we identify the kernel of the convolution operator, and the operator itself.)

It follows that the kernel of 
$  M_{1,Q} f $ is 
\begin{equation}\label{e:Ram}
\begin{split}
  M_{1,Q}  (n)  & =  \psi _{\tau Q/N} \ast d \sigma _{\tau } (n) 
\sum_{Q\leq q \leq 2Q}   \sum_{a \in \mathbb Z ^{\times } _{q}} e_q (a( \lvert  n\rvert ^2  - \lambda ^2 ) )
\\
 & =  P_{Q, \tau  } (n) \cdot \mathsf C _{Q} ( \lvert  n\rvert ^2  - \lambda ^2  ) . 
  \end{split}  
\end{equation} 
Note that $ P _{Q, \tau }$ is a maximal average over annuli of outer radius $ \tau $, and width  
about $ \tau Q/N < \tau  $.  
The second term above is related to  Ramanujan sums, defined by 
\begin{equation} \label{e:RamDef} 
\mathsf c _{q } (m) =  \sum_{a \in \mathbb Z ^{\times } _{q}} e_q (a m ) , \qquad m \in \mathbb Z , 
\end{equation}
so that in \eqref{e:Ram},  $\mathsf C _{Q} = \sum_{Q\leq q \leq 2Q}\mathsf c _{q }$.  
 Ramanujan sums satisfy very good cancellation properties.  
The properties we will need are summarized in 

%%%%%%%%%%%%%%%%%%%%%%%%%%%%%% LEMMA LEMMA LEMMA
\begin{lemma}\label{l:Ram} These two estimates hold, for any $ k \in \mathbb N$, and $ \epsilon >0$, 
%%  ENUMERATE
\begin{enumerate}
\item For any $ Q$ and $ n$,  $ \lvert  \mathsf C _Q (n)\rvert \leq Q ^2  $.  

\item   There holds 
\begin{equation}\label{e:Csup}
 \max _{ 0<  m \leq Q ^{k}}  \lvert  \mathsf C _{Q} (m)\rvert  \lesssim Q ^{1 + \epsilon }. 
\end{equation}

\item  For   $ M > Q ^{k}$, 
\begin{equation} \label{e:Ram<}
  \Biggl[  \frac{1}M \sum_{ m \leq M}  \lvert  C _{Q} (m)\rvert ^{k} 
  \Biggr] ^{1/k}
\lesssim   Q ^{1 + \epsilon }   .   
\end{equation}

\end{enumerate}
%% ENUMERATE
The implied constants depend upon $ k$ and $ \epsilon $.   
\end{lemma}
%%%%%%%%%%%%%%%%%%%%%%%%%%%%%% LEMMA LEMMA LEMMA

%%%%%%%%%%%%%%%%%%%%%%%%%%%%%% PROOF PROOF PROOF
\begin{proof} 
The first estimate is trivial, but we include it for the sake of clarity. Note that $ C _{Q} (0) \simeq Q ^2 $, which fact will arise in the absorption argument below.

An argument for the  second inequality  \eqref{e:Csup} begins with the 
inequality $ \lvert  \mathsf c _{q} (m)\rvert  \leq (q,m)$, for $ m>0$. 
This can be checked by inspection if $ q$ is a power of a prime. The general case follows 
as both sides are multiplicative functions.  

Then, of course we have for any $ 1\leq d \leq q$, 
\begin{equation*}
\sum_{k \;:\; dk \leq Q} d \leq Q.  
\end{equation*}
It follows that 
\begin{equation*}
\sum_{q \leq Q} (q,m) \leq Q \sum_{d  \leq q\;:\; d |m } 1= Q \delta (m; Q), 
\end{equation*}
where $ \delta (m;Q)$ is the number of divisors of $ m$ that are less than or equal to $ Q$. 
But, $ m\leq Q ^{k}$, so by a well known logarithmic type estimate for the divisor function, 
we have \eqref{e:Csup}. 

\smallskip 

The third property is harder. It is due to Bourgain \cite{MR1209299}.
There are proofs in   \cite{181012344}*{Lemma 2.13} and  \cite{180906468}*{Lemma 5}. 
\end{proof}
%%%%%%%%%%%%%%%%%%%%%%%%%%%%%% PROOF PROOF PROOF

%%%%%%%%%%%%%%%%%%%%%%%%%%%%%% SUBSUBSECTION SUBSUBSECTION SUBSUBSECTION 
\subsubsection*{A Tertiary Decomposition}%\label{sss.}
The preparations are finished.  It remains to prove \eqref{e:M1Q}, and this argument is indicated in Figure~\ref{f:teriary}. 
There are three cases, namely 
%%  ENUMERATE
\begin{enumerate} \setlength\itemsep{.75em}
\item   $ Q < N ^{1/2}$. 

\item  $ N Q ^{k_1-1} < \tau $, where $ k_1 = k_1 ( \bar p , \bar q)$. 

\item  $ N ^{1/2} \leq Q$ and $ \tau \leq   N Q ^{k_1-1}  $, implying $ N < \tau < Q ^{k_2}$, where $ k_2 = k_2 (\bar p, \bar q)$. 
\end{enumerate}
%% ENUMERATE
We treat these cases in order, with the core case being the last one.  

\smallskip 

The first, and easiest case, concerns $ Q < N ^{1 /2}$. Using the trivial bound $\lvert   \mathsf c_q (n)\rvert \leq  q $, 
and using \eqref{e:Ram} we have 
\begin{equation*}
\lvert  M _{1,Q} f \rvert < N \cdot P _{Q, \tau } f   .  
\end{equation*}
It then follows from the continuous sparse bound \eqref{e:R} that we have 
\begin{equation} \label{e:1M1Q}
 \lvert  E\rvert ^{-1} \langle M _{1,Q} f ,g  \rangle \lesssim  N \langle f  \rangle _{E} ^{1/ \bar p} \langle g \rangle ^{1/\bar q}.   
\end{equation}
This is as required in \eqref{e:M1Q}. 

\smallskip

%%%%%%%%%%%%%%% Figure
\begin{figure} 
\begin{tikzpicture}[sibling distance=12em, level distance=5.5em, 
    edge from parent/.style = {draw, -latex, font=\scriptsize}]
  \node [treenode] {$\displaystyle  \sum_{Q\leq q \leq  2Q} \sum  _{a \in \mathbb Z _q ^{\times }}  e_q (-a \lambda ^2 )   C ^{a/q,1} _{\tau } f   $}
    child  { node [treenode]{$  Q < N ^{1 /2}$} 
        edge from parent node [left] {\eqref{e:1M1Q} \ }} 
    child { node [treenode]{ $ N ^{1 /2} < Q$ \& $   N  < \tau  < Q ^{k_2}$}
      child { node [treenode]{ $  P _{\tau ,Q} ^0$}  
       edge from parent node [left] { \eqref{e:PP00} \ \ }
       edge from parent node [right] {absorb}} 
        child { node [treenode]{$   P _{\tau ,Q} ^1$} 
         edge from parent node [right] {\eqref{e:PP11} }} 
         child { node [treenode]{$   P _{\tau ,Q} ^2$} 
          edge from parent node [right] {\ \  \eqref{e:PP22} } } }
        child  { node [treenode]{$  \tau >  N Q ^{k_1-1} $} 
     edge from parent node [right] {\ \eqref{e:zM12} }} ; 
% 		edge from parent node [left] {$M_2$ \ \ }
%     edge from parent node [right] {\ \eqref{e:zM22}}}   
\end{tikzpicture}
\caption{The flow of the proof of  \eqref{e:M1Q}.  The integers $k_1$ and $k_2$ are large, and a function of $\epsilon$, $\bar p$ and $\bar q$. 
The first level of the decomposition is motivated by the estimates for Ramanujan sums in Lemma~\ref{l:Ram}.  The second level of the diagram is associated with the  Ramanujan estimate \eqref{e:Csup}.  It requires  a further decomposition of the kernel $ P _{\tau ,Q}$ in \eqref{e:Ram}. 
}  
\label{f:teriary}
\end{figure}
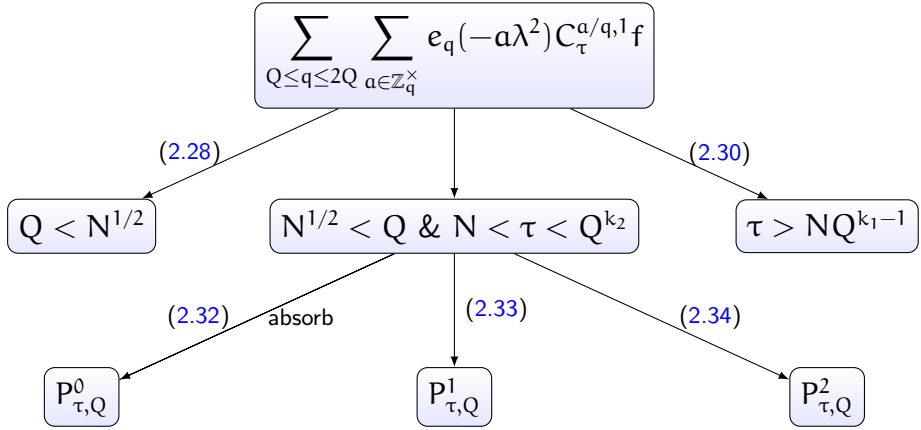
%%%%%%%%%%%%%%% Figure

The second case we restrict to the case that  $ \tau >  N Q ^{k_1-1}$, 
for a sufficiently large integer $ k_1$ that is a function of $ (\bar p, \bar q)$. 
This case does not have an absorbing term.    

Dominate, using  H\"older's inequality with $ \ell ^{k_1}$---$\ell ^{k_1'}$ duality, 
\begin{align} 
\lvert  M_{1,Q} f (n)\rvert 
& \lesssim  \bigl[  P_{Q, \tau  } \ast f (n) \bigr] ^{1/k_1'} 
\Bigl[  \sum_{x\in \mathbb Z ^{d}}   \lvert  C _{Q} (\lvert  x\rvert ^2 - \lambda ^2  )\rvert ^{k_1} P_{Q, \tau  } (x)\ \Bigr] ^{1/k_1} 
\\ \label{e:MM} 
& \lesssim Q ^{1+ \epsilon /2}    P_{Q, \tau  } \ast f (n)  ^{1/k_1'} . 
\end{align}
Recall that $ f$ is an indicator function.  
Notice that we are using a bound on the Ramanujan sums that follows from \eqref{e:Ram<}. 
Recall that $ P _{Q, \tau }$ is an average over an annulus around the sphere of radius $ \lambda $, 
of width $ \tau Q/N $.  In particular, the width is greater that $ Q ^{k_1}$, by assumption that 
$ \tau > N Q ^{k_1-1}$.

But, then, we are free to conclude our statement, since $ (1/\bar p, 1/\bar q)$ are in the interior of $ \mathbf R_d$,  
we have, for $ k_1$ sufficiently large, so that $ k'_1$ is sufficiently close to $ 1$,   
as required in \eqref{e:M1Q}, 
\begin{align}
\lvert  E\rvert ^{-1}  \langle  \lvert  M_1 f \rvert  , g\rangle 
& \lesssim Q ^{1+ \epsilon } \lvert  E\rvert ^{-1} \bigl\langle 
[P_{Q, \tau  } \ast f ] ^{1/k_1'},g 
\bigr\rangle  
\\
& \lesssim  N 
\bigl[ \lvert  E\rvert ^{-1} \bigl\langle 
P_{Q, \tau  } \ast f , g 
\bigr\rangle  \bigr] ^{1/k_1'} 
\\ & \lesssim  \langle  f \rangle _{ E} ^{ \frac 1 { \bar p k_1'}} \langle g \rangle _{E} ^{\frac 1 {  \bar q k_1'}}.  
\label{e:zM12} 
\end{align}
Here, we use \eqref{e:MM} and then the real variable inequality  \eqref{e:R}, which we can do if $ k_1$ is sufficiently large, so that $ (k_1'\bar p, k_1'\bar q) \in \mathbf R _d$.  
This is our second application of \eqref{e:R}.
 
\smallskip

We  turn to third case of $ N < \tau  < Q ^{k_2}$.  
A final, fourth decomposition of $ P _{\tau ,Q}$ is needed, and the absorption argument appears. 
Let  $ \mathbb S _{\lambda } = \{n \in \mathbb Z ^{d} \;:\; \lvert  n\rvert = \lambda  \}$ be the 
integer sphere of radius $ \lambda $, and set 
\begin{align}\label{e:P123}
P _{\tau ,Q }  &= \sum_{j=0} ^{2} P _{\tau ,Q} ^{j}, 
\\   \label{e:P0}
\textup{where} \qquad P _{\tau , Q} ^{0} (n) &= P _{\tau ,Q } (n) \mathbf 1_{ \mathbb S _{\tau  } } (n) , 
\\  \label{e:P1}
\textup{and} \qquad  P _{\tau , Q} ^{1} (n) & =  P _{\tau ,Q } (n) \mathbf 1_{ 0< \textup{dist} (n, \mathbb S _{\tau } ) < \tau Q^{ 1+ \epsilon }/N } ,  
\end{align}
and $P _{\tau , Q} ^{2} $ is then defined.  The term $ P _{\tau , Q} ^{0} $ is a multiple of the average over 
the integer sphere of radius $ \lambda $, and  the term $ P _{\tau , Q} ^{1}$ is just that part of $ P _{\tau , Q}$ that is close to, but not equal to, the sphere of radius $ \mathbb S _{\tau }$.

Let us detail the absorbing term.  Note that $ \mathsf C _{Q} (0) \simeq  Q ^2 $. 
Indeed, we have to single out this case as there is no cancellation in the Ramanujan sum, 
when the argument is zero.  
Using the definition \eqref{e:P0},  we have 
\begin{align}
\lvert  P _{\tau , Q} ^{0} (n) \cdot \mathsf C _{Q} ( \lvert  n\rvert ^2 - \tau ^2  ) \rvert 
& \lesssim    
Q ^2      \frac{N} { Q \tau ^ d } \mathbf 1_{ \mathbb S _{\tau }} (n) 
\\&   \label{e:PP00}
\lesssim   \frac{N  Q } {\tau ^2} \cdot  \tau ^{2-d}  \mathbf 1_{ \mathbb S _{\tau }} (n)  
\lesssim  N ^{-2 \epsilon}  \cdot  \tau ^{2-d}  \mathbf 1_{ \mathbb S _{\tau }} (n) . 
\end{align}
This is as required in \eqref{e:absorb} and \eqref{e:M1Q}.

The inequality \eqref{e:Csup} on the Ramanujan sums applies in the analysis of $ P _{\tau , Q} ^1$, 
due to our assumptions $  Q  < \tau  < Q ^{k_2}$.  
It shows that 
\begin{equation*}
\lvert  P_{Q, \tau } ^{1}  (n) \cdot \mathsf C _{Q} ( \lvert  n\rvert ^2 - \tau ^2  ) \rvert 
\lesssim Q ^{1+ \epsilon }\frac{ N } {\tau ^{d} Q}  
\mathbf 1_{ 0< \textup{dist} (n, \mathbb S _{\tau } ) < \tau Q^{ 1+ \epsilon }/N }
\end{equation*}
Keeping normalizations in mind, 
it  follows from our real variable sparse inequality \eqref{e:R} that 
\begin{equation}  \label{e:PP11}
\lvert  E\rvert ^{-1}  
\langle P_{Q, \tau } ^{1} \ast f , g \rangle \lesssim Q ^{1+ 2\epsilon }  \langle f \rangle_E ^{1/\bar p} 
\langle g \rangle _{E} ^{1/\bar q} .  
\end{equation}
This is as required in \eqref{e:M1Q}.  

The last term is $ P_{Q, \tau } ^2 $.  As noted,  $ \lvert  \mathsf C _{q} (m)\rvert \leq Q ^2  $. 
The condition $ \tau < Q ^{k_2}$, and simple Schwartz tail considerations then show that 
\begin{equation} \label{e:PP22}
\lvert  P_{Q, \tau } ^2  (n)  \cdot \mathsf C _{Q} ( \lvert  n\rvert ^2 - \tau ^2  )\rvert \lesssim    \frac{Q} {\tau ^{d} }  [ 1+ \lvert  n\rvert/ \tau ] ^{-2d} .  
\end{equation}
Since $ \tau $ is an admissible stopping time, it follows that 
\begin{equation*}
\lvert   P_{Q, \tau } ^{1} \ast f  \rvert \lesssim \langle f \rangle _{E}.   
\end{equation*}
This completes the proof of \eqref{e:M1Q}.

%%%%%%%%%%%%%%%%%%%%%%%%%%%%%% SECTION  SECTION SECTION
%%%%%%%%%%%%%%%%%%%%%%%%%%%%%% SECTION  SECTION SECTION 
\section{The Endpoint Sparse Bound } %\label{s:}

We need this definition.  Given cube $ E$, we say that collection $ \mathcal Q _E$ of subcubes 
$ Q\subset E$ are \emph{pre-sparse} if the cubes $ \{ \frac{1}3 Q \;:\; Q\in \mathcal Q_E\}$ are 
pairwise disjoint.  Associated to to a pre-sparse collection $ \mathcal Q_E$  
are a family of stopping times. We say that $ \tau $ is $ \mathcal   Q_E $ admissible  (or just admissible) if 
\begin{equation} \label{e:tau}
 \ell (E) \geq \tau _{\mathcal{Q}_E} (x) = \tau (x) \geq  \max \{  1 ,   \ell (Q) \mathbf 1_{\frac{1}3 Q} \;:\; Q\in \mathcal Q_E\}, \qquad x\in E. 
\end{equation}
The relevant Lemma is this. 

%%%%%%%%%%%%%%%%%%%%%%%%%%%%%% LEMMA LEMMA LEMMA
\begin{lemma}\label{l:main} For $ f = \mathbf 1_{F}$ supported on cube $ 3E$, 
there is a pre-sparse collection $ \mathcal Q _{E}$ so that 
for all  $ \mathcal Q_E $ admissible $ \tau = \tau (x)$,  and all  $ g= \mathbf 1_{G} $ 
supported on $ E$, we have 
\begin{equation} \label{e:main}
 \langle  A _{\tau  } f ,  g \rangle 
\lesssim \langle f \rangle _{3E, \frac{d} {d-2} } \langle g \rangle _{E , \frac{d} {d-2}} \lvert  E\rvert.  
\end{equation}
\end{lemma}
%%%%%%%%%%%%%%%%%%%%%%%%%%%%%% LEMMA LEMMA LEMMA

  The Lemma follows from this: For integers $ N >1$, there is a decomposition 
\begin{align}\label{e:goal}
  A _{\tau  } f  
 & \leq M_1+ M_2, 
 \\   \label{e:goal1}
 \langle M_1  , g\rangle &\lesssim N ^2  \langle f \rangle _{3E,1} \langle g \rangle _{E,1} \lvert  E\rvert , 
 \\ \label{e:goal2}
\textup{and} \qquad  \langle M_2  , g\rangle &\lesssim N ^{- \frac{d-4}2} \langle f \rangle _{3E,2 } \langle g \rangle _{E,2} \lvert  E\rvert .  
 \end{align}
Recalling that $ f = \mathbf 1_{F}$ and $ g= \mathbf 1_{G}$, the right sides above are comparable for 
$ N \simeq \bigl[\langle f \rangle _{3E,1} \langle g \rangle _{E,1} \bigr]^{- \frac{1}d}$, and this proves  \eqref{e:main}.  
\bigskip 

\tikzset{
 treenode/.style = {shape=rectangle, rounded corners,
    draw, align=center,
    top color=white, bottom color=blue!10}}

%%%%%%%%%%%%%%% Figure
\begin{figure} 
\begin{tikzpicture}[sibling distance=13em, level distance=5.5em, 
    edge from parent/.style = {draw, -latex, font=\scriptsize}]
  \node [treenode] {$ \displaystyle  A _{\tau  } f $}
    child  { node [treenode]{$ \displaystyle  \mathbf 1_{ \tau (x) \leq  N}  A _{\tau   } f  $} 
     edge from parent node [left] {$M_1$ \ }} 
    child { node [treenode]{$  \displaystyle A _{\lambda } = C _{\lambda } + E_{\lambda }$}
      child { node [treenode]{ $\displaystyle C _{\lambda} = \sum _{1\leq q \leq \lambda } \sum _{a \in \mathbb Z _q ^{\times }} C^{a/q}_\lambda $}
        child { node [treenode]{ $\displaystyle   \sum_{ N/100\leq q \leq \lambda } \sum _{a \in \mathbb Z _q ^{\times }}   C ^{a/q} _{\tau } f  $}  
 		edge from parent node [left] {$M_2$ \ \ }}   
        child { node [treenode] { $\displaystyle  C ^{a/q} _{\lambda } =  C ^{a/q,1} _{\lambda } +  C ^{a/q,2} _{\lambda } $} 
        child { node [treenode] {$\displaystyle  \sum_{1 \leq q \leq  N/100} \sum  _{a \in \mathbb Z _q ^{\times }}   C ^{a/q,1} _{\tau } f   $}
 		edge from parent node [left] {$M_1$ \ }} 
 		child { node [treenode] {$\displaystyle    \sum_{1\leq q \leq N/100} \sum _{a\in \mathbb Z_q ^{\times}} 
 C ^{a/q,2} _{\tau } f  $}
   edge from parent node [right] {\ $M_2$}
        }}  }
      child { node [treenode] {$\displaystyle      E_{\tau } f   $}
       edge from parent node [right] {\ $M_2$}}   };
\end{tikzpicture}
\caption{The flow of the proof of  \eqref{e:goal}---\eqref{e:goal2}.  The notation is similar to Figure~\ref{f:Fulltree}.  
 }  
\label{f:short}
\end{figure}
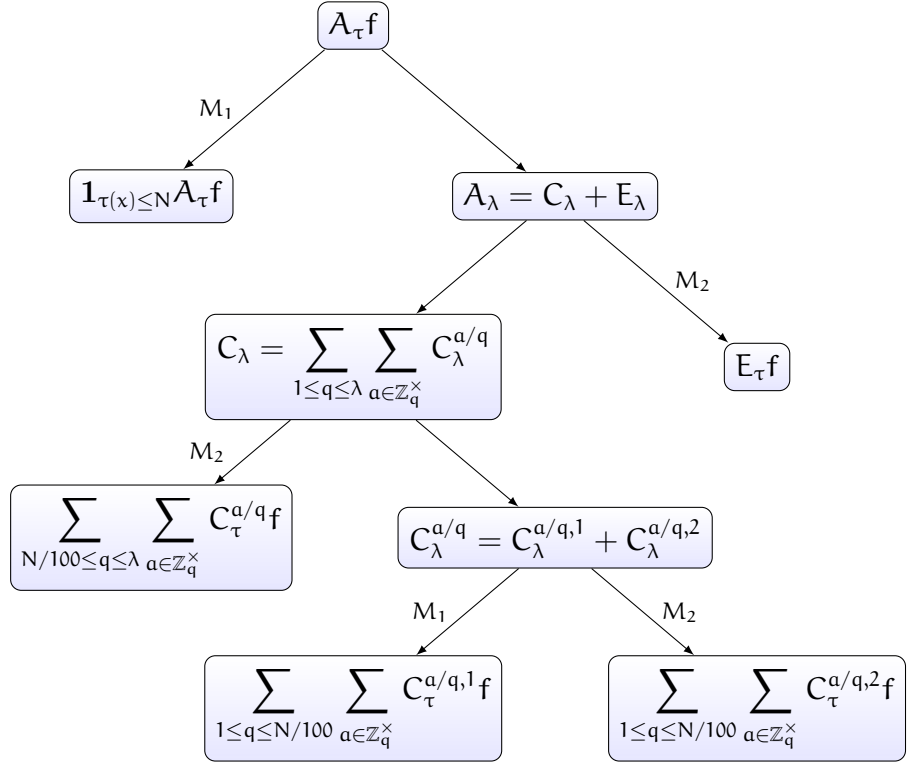
%%%%%%%%%%%%%%% Figure

It remains to prove \eqref{e:goal}---\eqref{e:goal2}.   Our proof will be much shorter, because we do not need to compare to the very rough continuous spherical averages, but to averages over balls.  (In particular, we will not need any subtle facts about Ramanujan sums.) 
A guide to the argument is in Figure~\ref{f:short}. 
The decomposition has several elements. The first begins with the trivial bound 
$
A _{\lambda } f (x) \lesssim \lambda ^2 B _{\lambda  } \ast f (x) 
$, where $ B _{\lambda }$ is the average of over a ball of radius $ \lambda $.  
Our first contribution to $ M_1 $ is,  
\begin{equation}\label{e:M11}
M _{1,1} =    \mathbf 1_{ \tau (x) \leq  100 N}  A _{\tau   } f , 
\end{equation}
which is pointwise bounded by $ C N ^{2}\langle f   \rangle _{3E,1}$, by choice of $ \mathcal Q_E$.  Thus \eqref{e:goal1} holds for this term.   Below, we are free to assume that $ \tau \geq  100N$.  

Recall the decomposition of $ A _{\lambda } f $, beginning with \eqref{e:Afull}. 
Our first contribution to $ M_2$ is 
$
M _{2,1}  =   \lvert  E_{\tau } f \rvert 
$. 
This  satisfies \eqref{e:goal2} by \eqref{e:zM21}.  

It remains to bound 
$
    C _{\tau }  f    
$ as defined in \eqref{e:Cfull}.  This  requires further contributions to $ M_1 $ and $ M_2$.  
Apply \eqref{e:MSWfactor} to see that this term obeys \eqref{e:goal2}. 
\begin{equation} \label{e:M22}
M_ {2,2} = 
 \sum_{ N/100 \leq q \leq \lambda } \sum _{a \in \mathbb Z _q ^{\times }}  \lvert   C ^{a/q} _{\tau } f \rvert  . 
\end{equation}

The remaining terms are 
\begin{equation} \label{e:smallq}
 \sum_{1 \leq q \leq  N/100} \sum _{a \in \mathbb Z _q ^{\times }}  e_q (- \lambda ^2 a) C ^{a/q} _{\tau  } f  . 
\end{equation}
Control will consist of an additional contribution to $ M_1$ and $ M_2$.

\bigskip 
Write   $ C _{\lambda } ^{a/q} = C_{\lambda } ^{a/q,1} +  C _{\lambda } ^{a/q,2} $, where a different cut-off in frequency is inserted.  
\begin{equation}\label{e:aq1}
 \begin{split}
 \widehat {C _{\lambda } ^{a/q,1}}  (\xi ) 
 &=  \sum_{\ell \in \mathbb Z ^{ d}} G (a/q, \ell ) \widetilde\psi_q (\xi - \ell /q)  \widetilde\psi _{\lambda q/N} (\xi - \ell /q) 
\widetilde{d \sigma _{\lambda }} (\xi - \ell /q)
\\ 
&=\sum_{\ell \in \mathbb Z ^{ d}} G (a/q, \ell )   \widetilde\psi _{\lambda q/N} (\xi - \ell /q) 
\widetilde{d \sigma _{\lambda }} (\xi - \ell /q) .
 \end{split}
\end{equation} 
This follows from the definition of $ \psi $ in \eqref{e:caq}.  
This is slightly different from \eqref{e:exC}, in particular, the term below satisfies \eqref{e:goal2}, 
just as in the previous section.  
\begin{equation*}
M _{2,3} = \Bigl\lvert  
 \sum_{1 \leq q \leq  N/100} \sum _{a \in \mathbb Z _q ^{\times }}  e_q (- \lambda ^2 a) C ^{a/q,2} _{\tau  } f
\Bigr\rvert. 
\end{equation*}

We claim that 
\begin{equation}\label{e:M12}
 \Bigl\lVert  
 \sum_{1 \leq q \leq  N/100} \sum _{a \in \mathbb Z _q ^{\times }}  e_q (- \lambda ^2 a) C ^{a/q,1} _{\tau } f  
 \Bigr\rVert _{\infty } \lesssim N ^2 \langle f \rangle _{E,1}. 
\end{equation}
That is, the term on the left is our second and final contribution to $ M_1$, as in \eqref{e:goal1}.

Recalling \eqref{e:mzeta} and \eqref{e:Ram}, we have 
\begin{align} \label{e:CC}
\bigl\lvert  {C_{\lambda } ^{a/q,1}} (n)\bigr\rvert &   
\lesssim  d \sigma _{\lambda  } \ast   {\psi _{\lambda q/N}}   (n) 
\end{align}
The convolution is with $ d \sigma _{\lambda }$ and $  {\psi _{\lambda q/N}}   $, which is a 
bump function of integral one, supported on scale $ N/ \lambda q$, which is much smaller than $ \lambda $. 
As a consequence, we have 
\begin{equation*}
\bigl\lvert  {C_{\lambda } ^{a/q,1}} (n)\bigr\rvert  
\lesssim \frac{N}q  \cdot   \lambda ^{-d}\Bigl[ 1+ \lvert  n\rvert/ \lambda   \Bigr] ^{- 3d}.  
\end{equation*}
This is summed over $ 1\leq  a < q \leq N /100$.  And, one appeals to the admissibility of the stopping time $ \tau $ to complete the proof of \eqref{e:M12}. 

%%%%%%%%%%%%%%%%%%%%%%%%%%%%%% SECTION  SECTION SECTION
%%%%%%%%%%%%%%%%%%%%%%%%%%%%%% SECTION  SECTION SECTION 
\section{Interpolation of Sparse Bounds} %\label{s:}

We show that if a sublinear operator satisfies an open range of sparse bounds for indicator sets, then 
they improve to sparse bounds for functions.  

%%%%%%%%%%%%%%%%%%%%%%%%%%%%%% LEMMA LEMMA LEMMA
\begin{lemma}\label{l:interpolate} Suppose that a sub-linear operator $ T$ satisfies the bound below 
for $ 1< p, q < \infty  $.  For a fixed function $ f$ and  all $ \lvert  g\rvert  \leq \mathbf 1_{G}$, 
there is a sparse collection $ \mathcal S$ so that 
\begin{equation*}
\lvert  \langle T f, g \rangle\rvert  \lesssim \Lambda _{\mathcal S, p, q} (f, \mathbf 1_{G})
\end{equation*}
Then,  
\begin{equation*}
\lvert  \langle T f, g \rangle\rvert  \lesssim   \sup _{\mathcal S}\Lambda _{\mathcal S, p, q} (f, g), 
\qquad q < r < \infty .  
\end{equation*}
\end{lemma}
%%%%%%%%%%%%%%%%%%%%%%%%%%%%%% LEMMA LEMMA LEMMA

%%%%%%%%%%%%%%%%%%%%%%%%%%%%%% PROOF PROOF PROOF
\begin{proof}
In this proof, we will work on $ \mathbb R ^{d}$, with the same proof working on $ \mathbb Z ^{d}$.  
We will also assume that  (1) $ T$ is a positive operator,  and (b) all cubes are dyadic. 
The general case is not much harder than  these considerations.  
Let $ f $, and let $ g$ be a bounded compactly supported function.  
We will show that there is a sparse collection $ \mathcal S$ so that 
\begin{equation}\label{e:genl}
\langle T f, g \rangle \lesssim \sum_{Q\in \mathcal S} \langle \mathbf 1_{F} \rangle _{Q} ^{1/p} 
\langle  g \rangle _{Q , q,1} \lvert  Q\rvert,  
\end{equation}
where $ \langle  g \rangle _{Q , q,1} = \lVert g \mathbf 1_{Q}\rVert _{L ^{q,1} (dx/ \lvert  Q\rvert )} $ 
is the Lorentz space with normalized measure.   Since $ \langle g \rangle _{Q, q,1} < \langle g \rangle _{Q, r}$, 
for $ 1 < r < q$, this completes the proof of the Lemma. 

The argument for \eqref{e:genl} is a level set argument. Thus, write  
$ g \leq \sum_{k\in \mathbb Z } 2 ^{k} \mathbf 1_{G_k}$, for disjoint sets $ G_k$.  
Apply the assumed sparse bound for indicators for each pair of sets $ (F, G_k)$. We get 
a sequence of sparse sets $ \mathcal S_k$ so that 
\begin{equation} \label{e:gen2}
\langle T f, g \rangle \lesssim \sum_{k \in \mathbb Z } \Lambda _{\mathcal S_k, p, q} (\mathbf 1_{F}, \mathbf 1_{G_k})  . 
\end{equation}
Let $ \mathcal F$ be a sequence of stopping cubes for the averages $ \langle f \rangle _{Q,p}$. 
That is, we choose $ \mathcal F$ so that for any dyadic cube $ Q$, there is a dyadic cube in $ \mathcal F$ 
that contains $ Q$.  And setting $ Q ^{a}$ to be the minimal such cube in $ \mathcal F$, we have $ \langle f \rangle_{Q,p} \lesssim 
\langle f \rangle _{Q ^{a},p}$. 
The sum in \eqref{e:gen2} is organized according to $ \mathcal F$.  
Below, we take $ q< \alpha < \infty $, very close to $ q$.  For each $ P \in \mathcal F$ we have 
\begin{align} 
\Gamma (P) &= 
\sum_{k\in \mathbb Z } 2 ^{-k} \sum_{\substack{Q\in \mathcal S_k\\   Q ^{a} =P }} 
\langle \mathbf 1_{G_k} \rangle _{Q} ^{1/q} \lvert  Q\rvert  
\\
& \lesssim 
\sum_{k\in \mathbb Z } 2 ^{-k} 
 \lvert  P\rvert ^{1/\alpha '}   \Bigl[  \sum_{\substack{Q\in \mathcal S_k\\   Q ^{a} =P }} 
\langle \mathbf 1_{G_k} \rangle _{Q} ^{\alpha /q} \Bigr] ^{1/ \alpha } 
\\  \label{e:gen3}
& \lesssim \lvert  P\rvert   \sum_{k\in \mathbb Z } 2 ^{-k} \langle \mathbf 1_{G_k} \rangle  _{P} ^{1/ q} 
= \langle g \rangle _{L ^{q,1} (P)} \lvert  P\rvert .  
\end{align}
Here, we have used H\"older's inequality, sparseness of the collections $ \mathcal S_k$  and the Carleson embedding inequality.  
The last inequality is one way to define the $ L ^{q,1}$ norm.  

And, so we have 
\begin{equation*}
\eqref{e:gen2}  \lesssim \sum_{P\in \mathcal F} \langle f \rangle _{P, p} \Gamma (P). 
\end{equation*}
The bound in \eqref{e:gen3} then implies \eqref{e:genl}.  
\end{proof}
%%%%%%%%%%%%%%%%%%%%%%%%%%%%%% PROOF PROOF PROOF

Concerning our Theorem~\ref{t:Z}. For the first assertion, we have proved a restricted weak type inequality 
for all $ (1/p, 1/q)$ in the interior of $ \mathbf Z_d$. 
We see that we can then replace the indicator functions in both coordinates by functions.  
That is, we have $ \langle A  \rangle _{p,q} < \infty $ for all $ (1/p, 1/q)$ in the interior of $ \mathbf Z_d$. 

Concerning the second assertion, we have the restricted weak type inequality at $ (\frac d{d-2},\frac d{d-2})$.  We see from the Theorem above that we have this consequence 
\begin{equation*}
\langle A f, g \rangle \lesssim 
\sup _{\mathcal S} \sum_{Q\in \mathcal S} \langle f \rangle _{Q, \frac d{d-2},1} \langle g \rangle _{Q,q} \lvert  Q\rvert , 
\qquad   q > \frac d{d-2}. 
\end{equation*}

%%%%%%%%%%%%%%%%%%%%%%%%%%%%%% SECTION  SECTION SECTION
%%%%%%%%%%%%%%%%%%%%%%%%%%%%%% SECTION  SECTION SECTION 
\section{Counterexamples} \label{s:counter}

We can show this proposition, showing necessary conditions on $ (p,q)$ for the sparse bound to hold.  The gap between the sufficient conditions for 
a sparse bound and these necessary conditions is illustrated in Figure~\ref{f:nec}.

%%%%%%%%%%%%%%% Figure
\begin{figure}[h]
\begin{tikzpicture}
\filldraw[color=black!10]  (0,3)  -- (2.7, 2.4) -- (2.7,0) -- (2.5,1) -- (0,3) ;
\draw [->]  (-.5, 0)  node [left] {$ \frac{d-2}d$}-- (5.5,0) node[below] {$ 1/p$}; 
\draw [->]  (0,-.5) -- (0,3.5) node[left] {$ 1/q$}; 
\draw (-.25 , 3 )  node [left] {$ 1$}  --  (0,3) --  (2.5,1) node[left] {$ {}_{ Z_2}$} -- (2.7, 0)  node [below] {$ Z_1$} ; 
\draw[dotted] (0,3)  -- (2.7, 2.4) -- (2.7,0); 
\draw[->] (5,1.5) node[below] {$ {}_{(\frac{d-2}d, \frac{d ^2 -2d+4} {d (d-1)})}$} -- (2.7,2.4); 
\draw[->] (5,3)   node[above] {$  \tfrac2{p} +\tfrac{ d-1}q = d$} -- (1.35,2.7) ; 
\end{tikzpicture}

\caption{The horizontal line is set at $ \frac{1}q =  \frac{d-2}d$, for reasons of clarity.  
Sparse bounds hold below the solid line from $ (0,1)$, to $ Z_2$ to $ Z_1= (\frac{d-2}d, \frac{d-2}d)$.  
(Recall that $ Z_2$ is defined in \eqref{e:Quad}.)
They cannot hold to the right of $ Z_1 $, nor above the dotted line. 
The gray area is unresolved.  
}
\label{f:nec}
\end{figure}
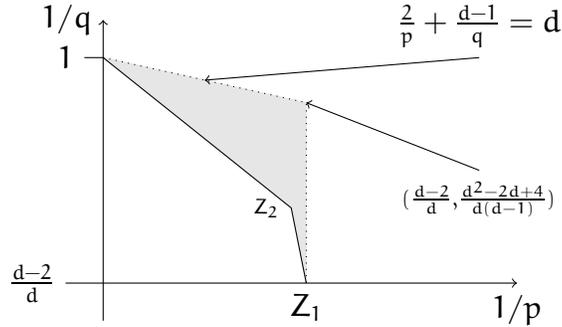
%%%%%%%%%%%%%%% Figure

%%%%%%%%%%%%%%%%%%%%%%%%%%%%%% PROPOSITION PROPOSITION PROPOSITION
\begin{proposition}\label{p:nec} If the sparse bound $ \lVert A \rVert _{ p,q} < \infty $ holds, we have 
\begin{equation}\label{e:nec}
 \frac1{p_1} \leq \frac{d-2}d, \qquad \frac2{p} +\frac{ d-1}q \leq d.  
\end{equation}
\end{proposition}
%%%%%%%%%%%%%%%%%%%%%%%%%%%%%% PROPOSITION PROPOSITION PROPOSITION

%%%%%%%%%%%%%%%%%%%%%%%%%%%%%% PROOF PROOF PROOF
\begin{proof}
If a  $ (p,q)$ sparse bound held we would conclude that $ A \;:\; \ell ^{p} \to \ell ^{p, \infty } $. 
Taking $ f = \mathbf 1_{0}$, note that $ A f (x) \simeq (1+\lvert  x\rvert)^{2-d} $. The latter function is $ \ell ^{ \frac{d} {d-2}, \infty }$. 
Hence, $ p_1 \geq \frac{d} {d-2}$ is necessary.  

For the second inequality in \eqref{e:nec},  set $ \mathbb S _{\lambda } = \{n\in \mathbb Z ^{d} \;:\;  \lvert  n\rvert  = \lambda  \}$.  
We recall that for $ d \geq 4$, and any odd choice of $ \lambda ^2 \in \mathbb N $, we have $ \lvert  \mathbb S _{\lambda }\rvert \simeq \lambda ^{d-2} $. 
For an odd choice of $ \lambda ^2 \in \mathbb N $, let  $ f =  \mathbf 1_{ \mathbb S _{\lambda }}$, 
and consider $  G = \{ A f > c/ \lambda \}$.  This is the set of $ x\in \mathbb Z ^{d}$ for which 
the two spheres $ \mathbb S _{\lambda } $ and $ x+ \mathbb S _{\mu }$, for some choice of $ \mu \simeq \lambda $,  have about the expected size.  
Here, necessarily   $ G\subset E$, 
a cube centred at the origin of side length about $ \lambda $.  

We claim that $ \lvert  G\rvert \gtrsim \lambda  $.  And observe that $ A f (x) \gtrsim \lambda ^{-1} $ for $ x \in G$. 
Moreover,  from the assumed 
 $ (p,q)$ sparse bound,  
\begin{align}  \label{e:NEC}
\lambda ^{-1} \langle \mathbf 1_{G} \rangle_E 
 &\lesssim \lambda ^{-d}  \langle A f, g   \rangle 
 \\&\lesssim \langle  f\rangle_E ^{\frac1{p}} \langle \mathbf 1_{G} \rangle _E ^{\frac1{q}} \lesssim \lambda ^{- \frac{2}p} \langle \mathbf 1_{G} \rangle _E ^{\frac1{q}}. 
\end{align}
For this pair of functions, it is easy to see that the maximal sparse form is the expression on the right.  
Our lower bound on the size of $ G$ proves the proposition.  

Note that since dimension $ d \geq 5$, and $ \lambda ^2 $ is odd, so there are about $ \lambda ^{d-3}$ choices of vectors $ y = (0, y_2 ,\dotsc, y_d) \in \mathbb S _{\lambda }$. 
(We insist on $ \lambda ^2 $ being odd to capture the case of $ d=5$ here.) 
Then,  if $ \lvert  x_1\rvert \leq \lambda /2  $,  note that   
\begin{equation*}
\lVert (0, y_2 ,\dotsc, y_d) - (x_1 , 0 ,\dotsc, 0)\rVert= \sqrt {\lambda ^2 + x_1 ^2 }  = \lambda'. 
\end{equation*}
From that, it follows that 
\begin{equation*}
A _{\lambda'  } f (x_1 ,0 ,\dotsc, 0) \simeq \lambda ^{-1}.  
\end{equation*}
This shows that $ \lvert  G\rvert \gtrsim \lambda  $. 
\end{proof}
%%%%%%%%%%%%%%%%%%%%%%%%%%%%%% PROOF PROOF PROOF

Observe that the $ (\frac d{d-2} , \frac d{d-2})$ sparse bound and \eqref{e:NEC} 
imply that 
$
\lvert  G\rvert \lesssim \lambda ^{\frac{d+4}2}  
$.
Our lower bound is certainly not sharp.  What is the correct size of $ G$?

%\bibliography{radon,../sparse}	
\bibliographystyle{alpha,amsplain}	

% \bib, bibdiv, biblist are defined by the amsrefs package.

\end{document}